\documentclass[12pt,a4paper]{article}

\usepackage[dvips]{graphicx}
\usepackage{calc}
\usepackage{color}
\usepackage{amsmath}
\usepackage{amssymb}
\usepackage{amscd}
\usepackage{amsthm}
\usepackage{amsbsy}
\usepackage{delarray}
\usepackage[T1]{fontenc}
\usepackage{inputenc}
\usepackage{enumerate}

\begin{document}



\setlength{\parindent}{5mm} \setlength{\marginparwidth}{30mm}
\renewcommand{\leq}{\leqslant}
\renewcommand{\geq}{\geqslant}
\newcommand{\N}{\mathbb{N}}
\newcommand{\Z}{\mathbb{Z}}
\newcommand{\R}{\mathbb{R}}
\newcommand{\C}{\mathbb{C}}
\newcommand{\F}{\mathbb{F}}
\newcommand{\g}{\mathfrak{g}}
\newcommand{\h}{\mathfrak{h}}
\newcommand{\K}{\mathbb{K}}
\newcommand{\RN}{\mathbb{R}^{2n}}
\newcommand{\ci}{C^{\infty}}
\renewcommand{\S}{\mathbb{S}}
\renewcommand{\H}{\mathbb{H}}
\newcommand{\eps}{\varepsilon}

\newcommand{\notemarge}[1]{\marginpar{\scriptsize{#1}}}

\theoremstyle{plain}
\newtheorem{theo}{Theorem}
\newtheorem{prop}[theo]{Proposition}
\newtheorem{lemma}[theo]{Lemma}
\newtheorem{definition}[theo]{Definition}
\newtheorem*{notation*}{Notation}
\newtheorem*{notations*}{Notations}
\newtheorem{corol}[theo]{Corollaire}
\newtheorem{conj}[theo]{Conjecture}
\newtheorem{question}[theo]{Question}
\newtheorem*{question*}{Question}
\newenvironment{demo}[1][]{\addvspace{8mm} \emph{Proof #1.
    ---~~}}{~~~$\Box$\bigskip}

\newlength{\espaceavantspecialthm}
\newlength{\espaceapresspecialthm}
\setlength{\espaceavantspecialthm}{\topsep} \setlength{\espaceapresspecialthm}{\topsep}

\newenvironment{example}[1][]{\refstepcounter{theo} 
\vskip \espaceavantspecialthm \noindent \textsc{Example~\thetheo
#1.} }%
{\vskip \espaceapresspecialthm}

\newenvironment{remark}[1][]{
\vskip \espaceavantspecialthm \noindent \textsc{Remark.~}}
{\vskip \espaceapresspecialthm}

\def\Homeo{\mathrm{Homeo}}
\def\Diff{\mathrm{Diff}}
\def\Symp{\mathrm{Symp}}
\def\Fib{\mathrm{Fib}}
\def\Emb{\mathrm{Emb}}
\def\Id{\mathrm{Id}}
\newcommand{\norm}[1]{||#1||}

\title{On the geometry of the space of fibrations}
\author{Vincent Humili\`ere$^1$ and Nicolas Roy$^2$}
\normalsize \maketitle \footnotetext[1]{Ludwig-Maximilian Universität, Munich, Germany. Partially supported by the ANR,
project "Symplexe". Email: \texttt{vincent.humiliere@mathematik.uni-muenchen.de}} \footnotetext[2]{Humbolt universität zu Berlin, Germany. Email: \texttt{roy@math.hu-berlin.de}}

\abstract{We study geometrical aspects of the space of fibrations between two given manifolds $M$ and $B$, from the
point of view of Fréchet geometry. As a first result, we show that any connected component of this space is the base
space of a Fréchet-smooth principal bundle with the identity component of the group of diffeomorphisms of $M$ as total
space. Second, we prove that the space of fibrations is also itself the total space of a smooth Fréchet principal
bundle with structure group the group of diffeomorphisms of the base $B$.}

\section{Introduction and results}\label{section intro}

The aim of this paper is to study some geometrical properties of the space $\Fib(M,B)$ of all smooth fibrations
$\pi:M\to B$, with $M$ and $B$ smooth finite-dimensional manifolds. By "fibration" we always mean a locally trivial
fiber bundle. According to Ehresmann Theorem \cite{ehresmann}, $\Fib(M,B)$ is nothing but the space of all smooth
surjective submersions from $M$ to $B$. This space is known to be an open subset of the Fréchet manifold of all smooth
maps $\ci(M,B)$, provided $M$ is closed (see e.g. \cite{hamilton}, p. 85). Throughout the paper, $M$ (and $B$) will
always be assumed to be closed and we will study $\Fib(M,B)$ in the framework of Fréchet differential geometry (see
\cite{michor_book} or \cite{hamilton} for comprehensive introductions).

This work is actually part of a larger project dealing with the study of the space of Lagrangian fibrations of
symplectic manifolds, in order to derive applications to the theory of Hamiltonian completely integrable systems (this
is the topic of subsequent papers \cite{roy_11,humiliere_roy_1}). The present article consists of a preliminary study
of the general (non-Lagrangian) case.




\medskip
Suppose $\Fib(M,B)$ is non empty and let $\pi_0\in\Fib(M,B)$ be a fibration. A $\pi_0$-\emph{vertical} diffeomorphism
$\phi$ is a diffeomorphism of $M$ which lifts the identity of $B$, i.e., which satisfies
$$\pi_0\circ\phi=\pi_0.$$  We denote by $\Diff(M)$ the group of diffeomorphisms of $M$, by $\Diff_0(M)$ its identity
component and by $\mathcal{G}_{\pi_0}$ the subgroup of all $\pi_0$-vertical diffeomorphisms. We prove two independent
theorems on the geometry of $\Fib(M,B)$. The first one relates the diffeomorphism group of $M$ and $\Fib(M,B)$.

\begin{theo}\label{theorem diff/fib} Let $\pi_0\in\Fib(M,B)$.
\begin{enumerate}
\item The action of $\mathcal{G}_{\pi_0}$ (resp. $\mathcal{G}_{\pi_0}\cap\Diff_0(M)$) on $\Diff(M)$ (resp. $\Diff_0(M)$) by right composition gives $\Diff(M)$ (resp. $\Diff_0(M)$)
the structure of a Fréchet principal bundle.
\item The connected component of $\pi_0$ in $\Fib(M,B)$ is naturally Fréchet diffeomorphic to the quotient space
$\Diff_0(M)/(\mathcal{G}_{\pi_0}\cap\Diff_0(M))$.
\end{enumerate}
\end{theo}

The first and second parts of this theorem are proved respectively in Sections \ref{section diff/fib} and \ref{section
fibrations as fibers}.


\begin{remark}The action of $\mathcal{G}_{\pi_0}$ by left composition leads also to a principal bundle structure, naturally obtained from the
previous one by conjugation by the inversion of diffeomorphisms.
\end{remark}

\begin{remark}The diffeomorphism that we construct to prove the second part of Theorem \ref{theorem diff/fib} is
induced by the map
$$\Diff_0(M)\to\Fib(M,B),\quad\phi\mapsto\pi_0\circ\phi^{-1}.$$
Note that in general, the analogous map defined on $\Diff(M)$ is not surjective onto $\Fib(M,B)$, so that there is no
similar result without the connected components assumptions. Indeed, for some manifolds $M,B$ one can find two
fibrations $\pi_1,\pi_2:M\to B$ with non-diffeomorphic fibers, which therefore can not satisfy $\pi_1=\pi_2\circ\phi$
for any $\phi\in\Diff(M)$.
 For instance, let $M=\S_3\times\S_2$, $B=\S_2$, $\pi_1$ be the
canonical projection onto its second factor and $\pi_2$ be the composition of the Hopf fibration $\S_3\to\S_2$ with the
projection onto the first factor. The fiber of $\pi_1$ is then $\S_3$ while the fiber of $\pi_2$ is $\S_1\times\S_2$.

One immediate consequence of Theorem \ref{theorem diff/fib} is that $\Diff_0(M)$ acts transitively on each connected
component of $\Fib(M,B)$. This was already proved by Michor in \cite{michor} using the Nash-Moser implicit function
Theorem. On the contrary, our proof is based on explicit constructions.

A corollary of this transitivity property is the following lemma, for which we can also give a very simple and direct
proof:

\begin{lemma}Two fibrations $\pi_0$, $\pi_1$ lying in the same connected component of $\Fib(M,B)$ have
diffeomorphic fibers.
\end{lemma}
\begin{proof}Since $\Fib(M,B)$ is a smooth Fréchet manifold, one can find a smooth loop $\Pi$ in $\Fib(M,B)$ going through $\pi_0$ and
$\pi_1$. This loop then defines a map $$\hat{\Pi}:\S^1\times M\to \S^1\times B,\quad (s,x)\mapsto(s,\Pi(s)(x))$$ which
is a submersion. Indeed, its differential is everywhere upper-triangular with submersive diagonal blocks. Then,
according to Ehresmann Theorem \cite{ehresmann}, $\hat{\Pi}$ is a fibration and in particular its fibers are all
diffeomorphic to each other. In particular the fibers of $\pi_0$ are diffeomorphic to those of $\pi_1$.
\end{proof}

\end{remark}

The second theorem concerns the action of $\Diff(B)$ on $\Fib(M,B)$ by left composition. It was inspired by Michor's
article \cite{michor3} about the principal bundle structure of the space of embeddings.

\begin{theo}\label{theorem fib is fibered} Let $\pi_0\in\Fib(M,B)$. Suppose that $\pi_0$ admits a global section.
Then, the action of $\Diff_0(B)$ by left composition gives the connected component of $\pi_0$ in $\Fib(M,B)$ the
structure of a Fréchet $\Diff_0(B)$-principal bundle.
\end{theo}

We will prove this theorem in Section \ref{section fib is fibered}.

\begin{remark} Intuitively, two fibrations in the same orbit under the action of $\Diff_0(B)$ define the same foliation of $M$,
so that the quotient space can be viewed as the space of bundle-like foliations.\end{remark}

\begin{remark} Theorem \ref{theorem fib is fibered} still holds when one replaces $\Diff_0(B)$ by $\Diff(B)$ and the
connected component of $\pi_0$ by the subset of $\Fib(M,B)$ consisting in all fibrations admitting a global section.
This subset is a union of connected components as follows easily from Theorem \ref{theorem diff/fib}.
\end{remark}

\begin{remark}What happens when $\pi_0$ does not admit any global section remains an open problem.

Strangely enough, when $M$ is symplectic and we consider the smaller space of Lagrangian fibrations, we have been able
to prove the existence of the principal bundle structure (under the action of the whole $\Diff(B)$) without any
assumption on the existence of a global section \cite{humiliere_roy_1}. The methods used therein involve the Nash-Moser
Theorem but unfortunately do not apply here.
\end{remark}

\subsubsection*{Acknowledgements}A large part of this work has been done during our "Research in Pair" stay at the MFO research center in Oberwolfach in February 2009.
We thank the MFO for its warm hospitality and for the perfect working conditions it provided to us.

\section{Preparation Lemmas}\label{section lemmes prepa}

In this section, we prove two independent lemmas that will be useful in the proof of both Theorems \ref{theorem
diff/fib} and \ref{theorem fib is fibered}.

\begin{lemma}
\label{lem-voisinage-pratique}Let $p:X\rightarrow B$ be a fiber bundle and $Y\subset X$ a subbundle, i.e., a submanifold
of $X$ such that the restriction of $p$ to it defines a fiber bundle over $B$. Then, there exists a tubular
neighbourhood $U$ of $Y$, whose associated projection $q:U\rightarrow Y$ is $p$-vertical, i.e.,\[ p\circ q=p.\]
\end{lemma}
\begin{figure}[!h]
\centering
\includegraphics[scale=0.6]{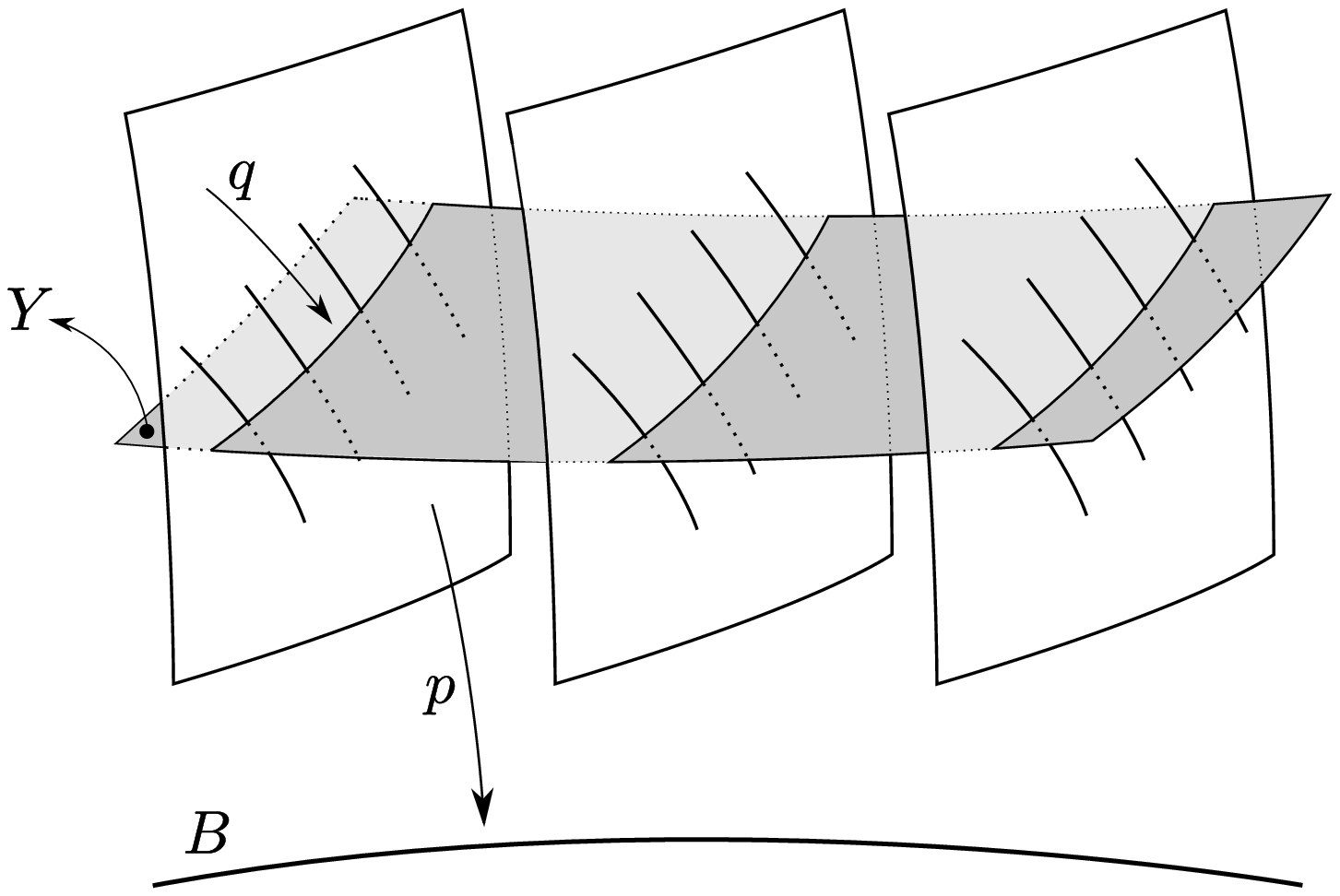}
\label{double_fibration}
\end{figure}

\begin{remark}
Notice that in the case where $X\rightarrow B$ is a vector bundle, one can construct this tubular neighbourhood with
the help of {}``local additions'', as defined in \cite{michor3}.\end{remark}
\begin{proof}
The usual construction of tubular neighbourhoods is the following. We first consider the normal bundle $N_{Y}$ to $Y$,
which is a vector bundle over $Y$. Then, with the help of a Riemannian metric on $X$, we can construct a diffeomorphism
$\phi$ from a neighbourhood of the $0$-section in $N_{Y}$ to a neighbourhood of $Y$ in $X$. Namely, one identifies any
element of $N_{Y}$ with a tangent vector on $X$ which is orthogonal to $Y$, and the exponential map provides a point in
$X$. Then, the projection $q$ corresponds through the diffeomorphism $\phi$ to the projection $N_{Y}\rightarrow Y$.
Unfortunately, constructed in this way, $q$ has no reason to be $p$-vertical.

Instead of fixing a metric on $X$ and using the corresponding exponential map to construct $\phi$, we rather take a
smooth family of metrics $\left(g_{b}\right)_{b\in B}$ on the fibers $X_{b}=p^{-1}\left(b\right)$. Such a family can be
obtained for example by restricting a given metric $g$ on $X$ to each fiber $X_{b}$. Then, at each point $x\in Y$, we
denote $b=p\left(x\right)$ and we identify $N_{Y}\left(x\right)$ with the vector space $W_{x}$ orthogonal to
$T_{x}\left(Y\cap X_{b}\right)$ in $T_{x}X_{b}$, with respect to the metric $g_{b}$ on $X_{b}$, simply by
\begin{eqnarray*}
W_{x} & \rightarrow & N_{Y}\left(x\right)=T_{x}X/T_{x}Y\\
V & \rightarrow & \left[V\right].\end{eqnarray*} This is injective since the intersection of $W_{x}$ with $T_{x}Y$ is
trivial. It is also surjective because of dimension matching. Indeed, the dimension of $N_{Y}\left(x\right)$ equals
$\dim X-\dim Y$. On the other hand, the dimension of $W_{x}$ is just $\dim X_{b}-\dim Y_{b}$, where $Y_{b}=Y\cap
X_{b}$. But this is equal to $\left(\dim X-\dim B\right)-\left(\dim Y-\dim B\right)$, hence to $\dim
N_{Y}\left(x\right)$. This defines therefore an isomorphism between $W_{x}$ and $N_{Y}\left(x\right)$.

The construction of the diffeomorphism $\phi$ is as follows. For any element in $N_{Y}\left(x\right)$ at some point
$x\in Y$, we take the corresponding vector in $W_{x}$. Then, we use the exponential map associated to the metric
$g_{b}$, where $b=p\left(x\right)$, and obtain a point which lies by construction in the same fiber $X_{b}$ as $x$
does. The smoothness of this map is clear. The fact that it is a diffeomorphism from a neighbourhood of the $0$-section
in $N_{Y}$ to a neighbourhood of $Y$ in $X$, follows from the property that the linearisation at $x$ of the exponential
map is simply the identity on $T_{x}X_{b}$. Finally, the fact that $q$ is $p$-vertical follows directly from the
construction of $\phi$.
\end{proof}

\begin{lemma}\label{lemma trivialisation->principal bundle}Let $\mathcal{F}$ be a Fréchet manifold together with an action of a Fréchet Lie group $\mathcal{G}$.
Suppose that there exists a Fréchet space $\mathcal{E}$, such that for any $f\in\mathcal{F}$, there exist a
$\mathcal{G}$-invariant neighbourhood $\mathcal{W}_f$ of $f$ in $\mathcal{F}$, an open set $\mathcal{V}_f$ in
$\mathcal{M}$ and a Fréchet diffeomorphism
$$\Phi_f:\mathcal{W}_f\to\mathcal{G}\times \mathcal{V}_f,$$ which is equivariant under the action of $\mathcal{G}$, i.e., for all
$g\in\mathcal{G}$, $\varphi\in \mathcal{W}_f$,
$$\Phi_f(g \cdot \varphi)=(g\Phi_f^1(\varphi),\Phi_f^2(\varphi)),$$
where $\Phi_{f}^1$, $\Phi_{f}^2$ denote the respective components of $\Phi_f$.

Then, $\mathcal{F}$ has the structure of a Fréchet principal $\mathcal{G}$-bundle.
\end{lemma}

\begin{proof}Let $\mathcal{Q}$ denote the quotient space of $\mathcal{F}$ under the action of $\mathcal{G}$, endowed with the quotient topology. Let us check that
$\mathcal{Q}$ is a Fréchet manifold.

For any $f\in\mathcal{F}$, the set $\mathcal{V}_f$ is then homeomorphic to some open set $\mathcal{U}_f$ in
$\mathcal{Q}$ and the family $(\mathcal{U}_f)_{f\in\mathcal{F}}$ covers $\mathcal{Q}$. Let us check that $\mathcal{Q}$
is Haussdorf. Let $q$, $q'$ be distinct elements of $\mathcal{Q}$. Let $\mathcal{U}$ denote one element of the family
$(\mathcal{U}_f)_{f\in\mathcal{F}}$ containing $q$. If $q'$ is in $\mathcal{U}$, since the Fréchet manifold
$\mathcal{M}$ is haussdorf, there are two disjoint open subsets in $\mathcal{U}$ (and thus in $\mathcal{Q}$),
containing respectively $q$ and $q'$. If $q'$ is not in $\mathcal{U}$, then we can also find two disjoint open sets
containing $q$ and $q'$ by taking any open neighbourhood of $q$ in $\mathcal{U}$ whose closure is included in
$\mathcal{U}$ and the complement of its closure in $\mathcal{Q}$. This is possible since a Fréchet topology is
metrizable.

Now, let $q=\Phi_f^2(f)$ and $q'=\Phi_{f'}^2(f')$ be two distinct elements in $\mathcal{Q}$, with two neighbourhoods
$\mathcal{U}_f$, $\mathcal{U}_{f'}$ such that $\mathcal{U}_f\cap \mathcal{U}_{f'}\neq\emptyset$, and such that there
exist two homeomorphisms $\phi:\mathcal{U}_f\to \mathcal{V}_f$, $\phi':\mathcal{U}_{f'}\to \mathcal{V}_{f'}$. Then, for
any fixed $g\in\mathcal{G}$, the transition map can be written
$$\phi'\circ\phi^{-1}|_{\phi(\mathcal{U}_f\cap \mathcal{U}_{f'})}=\Phi_{f'}^2\circ(\Phi_f)^{-1}|_{\{g\}\times\phi(\mathcal{U}_f\cap \mathcal{U}_{f'})}$$
and hence is smooth. Therefore, the family of sets $(\mathcal{U}_f)_{f\in\mathcal{F}}$ is a smooth Fréchet atlas for
$\mathcal{Q}$.

Finally, we see that the family of maps $(\Phi_f)_{f\in\mathcal{F}}$ are smooth local equivariant trivializations whose
second coordinate correspond to the natural projection $\mathcal{F}\to\mathcal{Q}$. We thus have a principal bundle
structure.
\end{proof}

\section{The principal bundle structure of $\Diff(M)$}\label{section diff/fib}

In this section, we prove that given a fixed fibration $\pi_{0}:M\rightarrow B$ the action \begin{eqnarray*}
\mathcal{G}_{\pi_{0}}\times\mbox{Diff}\left(M\right) & \longrightarrow & \mbox{Diff}\left(M\right)\\
\left(\psi,\phi\right) & \longmapsto & \phi\circ\psi\end{eqnarray*}
 gives $\mbox{Diff}\left(M\right)$ the structure of a Fréchet principal
bundle with structure group $\mathcal{G}_{\pi_{0}}$, as claimed in the first point of Theorem \ref{theorem diff/fib}.
We leave to the reader to check that the proof works if one replaces $\Diff(M)$ by $\Diff_0(M)$ and
$\mathcal{G}_{\pi_0}$ by $\mathcal{G}_{\pi_0}\cap\Diff(M)$.

The proof is divided in three steps:
\begin{itemize}
\item In Section \ref{sub:The-structure-group} we show that the orbits
of the $\mathcal{G}_{\pi_{0}}$-action are Fréchet submanifolds of $\mbox{Diff}\left(M\right)$ and in particular that
$\mathcal{G}_{\pi_{0}}$ is indeed a Fréchet Lie group.
\item Then, in Section \ref{sub:The-local-sections} we construct a $\mathcal{G}_{\pi_{0}}$-invariant
neighbourhood \emph{$\mathcal{U}$} of $\Id\in\mbox{Diff}\left(M\right)$ together with a Fréchet submanifold
$\mathcal{S}\subset\mathcal{U}$, transverse to the $\mathcal{G}_{\pi_{0}}$-orbits.
\item Finally, Section \ref{sub-Frechet-bundle-chart} provides the construction
of the local charts of the Fréchet principal bundle \[
\mbox{Diff}\left(M\right)\supset\mathcal{U}\longrightarrow\mathcal{S}\times\mathcal{G}_{\pi_{0}}.\] Thanks to the
transitive action of $\mbox{Diff}\left(M\right)$ onto itself by left composition, we then obtain a chart near any
$\phi\in\mbox{Diff}\left(M\right)$.
\end{itemize}
Throughout the proof, we will use intensively the standard identification of smooth maps on $M$ with smooth sections of
$M\times M$, namely
\begin{eqnarray*}
C^{\infty}\left(M,M\right) & \overset{\cong}{\longrightarrow} & \Gamma\left(M,M\times M\right)\\
\phi & \longmapsto & \hat{\phi}=\left(\Id,\phi\right).\end{eqnarray*} Here and always except
when stated explicitly, $M\times M$ is viewed as a trivial bundle over the first factor. Notice also that through this
identification, diffeomorphisms of $M$ correspond to an open subset of $\Gamma\left(M,M\times M\right)$, which we
denote by $\widehat{\mbox{Diff}}\left(M\right)$.

\subsection{\label{sub:The-structure-group}The structure group $\mathcal{G}_{\pi_{0}}$}

We construct a subbundle $N\subset M\times M$ over $M$ which will provide later a suitable principal bundle chart for
$\mbox{Diff}\left(M\right)$ around $\Id$. We define the subset $N$ by\[ N=\left\{ \left(x,y\right)\in M\times
M\mid\pi_{0}\left(x\right)=\pi_{0}\left(y\right)\right\} .\]

\begin{figure}[!h]
\centering
\includegraphics[scale=0.6]{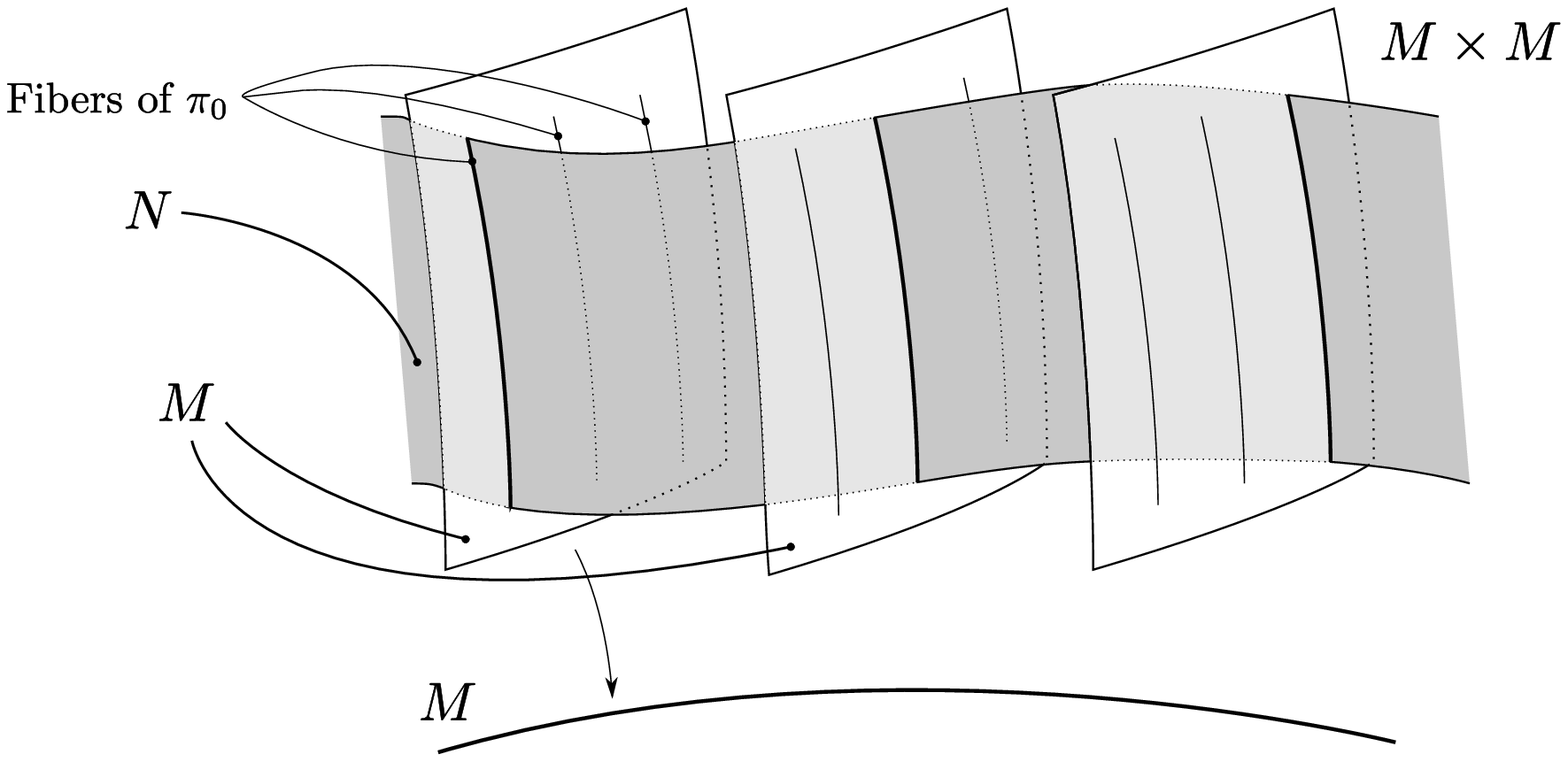}
\label{double_fibration}
\end{figure}

This is nothing but the pullback bundle of
$M\overset{\pi_{0}}{\rightarrow}B$ by the (same) smooth map $\pi_{0}:M\rightarrow B$, hence a subbundle of $M\times M$
over $M$. This subbundle $N$ provides a nice parametrisation of the $\mathcal{G}_{\pi_{0}}$-orbits because of the
following equivalence.
\begin{lemma}
\label{lem-section-N-equivalent-vert}Let $\phi\in\mbox{Diff}\left(M\right)$ be a diffeomorphism. Then the corresponding
$\hat{\phi}\in\Gamma\left(M,M\times M\right)$ lies in $\Gamma\left(M,N\right)$ if and only if
$\phi\in\mathcal{G}_{\pi_{0}}$.\end{lemma}
\begin{proof}
If $\phi\in\mathcal{G}_{\pi_{0}}$, then for each $x$, $\pi_{0}\circ\phi\left(x\right)$ equals $\pi_{0}\left(x\right)$
since $\phi$ is vertical with respect to $\pi_{0}$. But this precisely means that $\left(x,\phi\left(x\right)\right)$
lies in $N$ for all $x\in M$, i.e., $\hat{\phi}\in\Gamma\left(M,N\right)$. Conversely, if
$\left(x,\phi\left(x\right)\right)\in N$ for all $x\in M$, this means by definition that
$\pi_{0}\left(x\right)=\pi_{0}\circ\phi\left(x\right)$ for all $x$, hence $\phi\in\mathcal{G}_{\pi_{0}}$.
\end{proof}
It follows from this lemma that $\mathcal{G}_{\pi_{0}}$ is identified through the correspondance
$\phi\mapsto\hat{\phi}$ with the intersection of the open subset
$\widehat{\mbox{Diff}}\left(M\right)\subset\Gamma(M,M\!\times\! M)$ and $\Gamma\left(M,N\right)$. On the other hand, it
is well-known \cite[Exp 4.2.2]{hamilton} that the set of sections of a subbundle is a Fréchet submanifold of the set of
sections of the bundle. Therefore $\mathcal{G}_{\pi_{0}}$ is a Fréchet sub\-mani\-fold of $\mbox{Diff}\left(M\right)$.
Moreover $\mathcal{G}_{\pi_{0}}$ is also a subgroup of $\mbox{Diff}\left(M\right)$, which is a Fréchet Lie group. This
proves the following.
\begin{lemma}
The group $\mathcal{G}_{\pi_{0}}$ is a Fréchet Lie group.
\end{lemma}
Notice that the corresponding Lie algebra is $\Gamma\left(M,V_{\pi_{0}}\right)\subset\mathfrak{X}\left(M\right)$, where
$V_{\pi_{0}}\subset TM$ is the $\pi_{0}$-vertical tangent bundle of $M$, i.e., $V_{\pi_{0}}\left(x\right)=\ker
D\pi_{0}\left(x\right)$.

Notice also that the orbit of any $\phi_{0}\in\mbox{Diff}\left(M\right)$ is also a Fréchet submanifold of
$\mbox{Diff}\left(M\right)$. Indeed, this orbit is simply the image of $\mathcal{G}_{\pi_{0}}$ by the left composition
map $L_{\phi_{0}}:\mbox{Diff}\left(M\right)\rightarrow\mbox{Diff}\left(M\right)$ which sends any $\phi$ to
$\phi_{0}\circ\phi$. This map is smooth \cite[Exp. 4.4.5]{hamilton} and it inverse $L_{\phi_{0}^{-1}}$ as well. It is
therefore a Fréchet diffeomorphism of $\mbox{Diff}\left(M\right)$, and the result follows.

\subsection{\label{sub:The-local-sections}The local section $\mathcal{S}$}

We now need a tubular neighbourhood of $N$ in $M\times M$ with special properties, reflecting the fact that fibers of
$N$ over two different points $x,x'\in M$ satisfying $\pi_{0}\left(x\right)=\pi_{0}\left(x'\right)$ are identified, since both
are naturally identified with $\pi_{0}^{-1}\left(\pi_{0}\left(x\right)\right)$.
\begin{lemma}
\label{lem-tubular-neighbourhood-N}There exists a tubular neighbourhood $U\subset M\times M$ of $N$, which is invariant
under the action of $\mathcal{G}_{\pi_{0}}$ on the second factor and whose projection $P:U\rightarrow N$ has the form
\[ P\left(x,y\right)=\left(x,P_{2}\left(x,y\right)\right)\]
 with $P_{2}:U\rightarrow M$ satisfying $P_{2}\left(\psi\left(x\right),y\right)=P_{2}\left(x,y\right)$
for any $\psi\in\mathcal{G}_{\pi_{0}}$.\end{lemma}
\begin{proof}
In order to get the required $\mathcal{G}_{\pi_{0}}$-invariance property of our tubular neighbourhood inside $M\times
M$, we will first make a construction in $B\times M$ and then lift it to $M\times M$.

Let us consider the trivial bundle $B\times M$ over $B$. Similarly as above, one defines\[ \tilde{N}=\left\{
\left(b,y\right)\in B\times M\mid b=\pi_{0}\left(y\right)\right\} ,\] which is a subbundle of $B\times M$ over $B$. Its
fiber over $b$ is simply $\left\{ b\right\} \times\pi_{0}^{-1}\left(b\right)$. Now, according to Lemma
\ref{lem-voisinage-pratique}, we can construct a tubular neighbourhood $\tilde{U}\subset B\times M$ of $\tilde{N}$ such
that the fibers of its associated projection $\tilde{P}:\tilde{U}\rightarrow\tilde{N}$ are included in the fibers of
$B\times M$, i.e., $\tilde{P}$ has the form $\tilde{P}\left(b,y\right)=\left(b,\tilde{P}_{2}\left(b,y\right)\right)$.

On the other hand, one can assume that $\tilde{U}$ is invariant under the action of $\mathcal{G}_{\pi_{0}}$ on the
second factor of $B\times M$. Indeed, for any neighbourhood $V\subset B\times B$ of the diagonal, the set
$\rho^{-1}\left(V\right)$, where $\rho:B\times M\rightarrow M$ is defined by
$\rho\left(b,x\right)=\left(b,\pi_{0}\left(x\right)\right)$, is a neighbourhood of $\tilde{N}$ in $B\times M$. Then, we
can take $V$ so small that $\rho^{-1}\left(V\right)$ is contained in $\tilde{U}$. To see this, fix a metric on $B\times
M$ and consider the distance $\delta$ between $\tilde{N}$ and the boundary of $\tilde{U}$. It is non-vanishing by
compactness of $M$ and $B$. Then one can take $V$ with a diameter smaller than $\delta$, implying
$\rho^{-1}\left(V\right)\subset\tilde{U}$. In other word, up to taking $\tilde{U}$ smaller, we can assume it has the
form $\tilde{U}=\rho^{-1}\left(V\right)$, which is by construction invariant under the action of
$\mathcal{G}_{\pi_{0}}$ on the second factor of $B\times M$.

Now, if we define $\hat{\pi}_{0}:M\times M\rightarrow B\times M$ by
$\hat{\pi}_{0}\left(x,y\right)=\left(\pi_{0}\left(x\right),y\right)$, it follows that
$\hat{\pi}_{0}^{-1}\left(\tilde{N}\right)$ is precisely $N$ and that $U=\hat{\pi}_{0}^{-1}\left(\tilde{U}\right)$ is a
neighbourhood of $N$. Then we check easily that the map \[
P:\left(x,y\right)\longmapsto\left(x,\tilde{P}_{2}\circ\hat{\pi}_{0}\left(x,y\right)\right)\]
 is indeed a projection from $U$ onto $N$. %
Moreover, by construction for any $\psi\in\mathcal{G}_{\pi_{0}}$ and any point $x\in M$,
one has $\pi_{0}\left(x\right)=\pi_{0}\left(\psi\left(x\right)\right)$. This implies that \[
\tilde{P}_{2}\circ\hat{\pi}_{0}\left(\psi\left(x\right),y\right)=\tilde{P}_{2}\circ\hat{\pi}_{0}\left(x,y\right)\]
 and therefore $P_{2}\left(\psi\left(x\right),y\right)=P_{2}\left(x,y\right)$.
On the other hand, it is straightforward to check that $U$ is invariant under the action of $\mathcal{G}_{\pi_{0}}$  on
the second factor of $B\times M$, because $\tilde{U}$ is so.
\end{proof}
\begin{figure}[!h]
\centering
\includegraphics[scale=0.6]{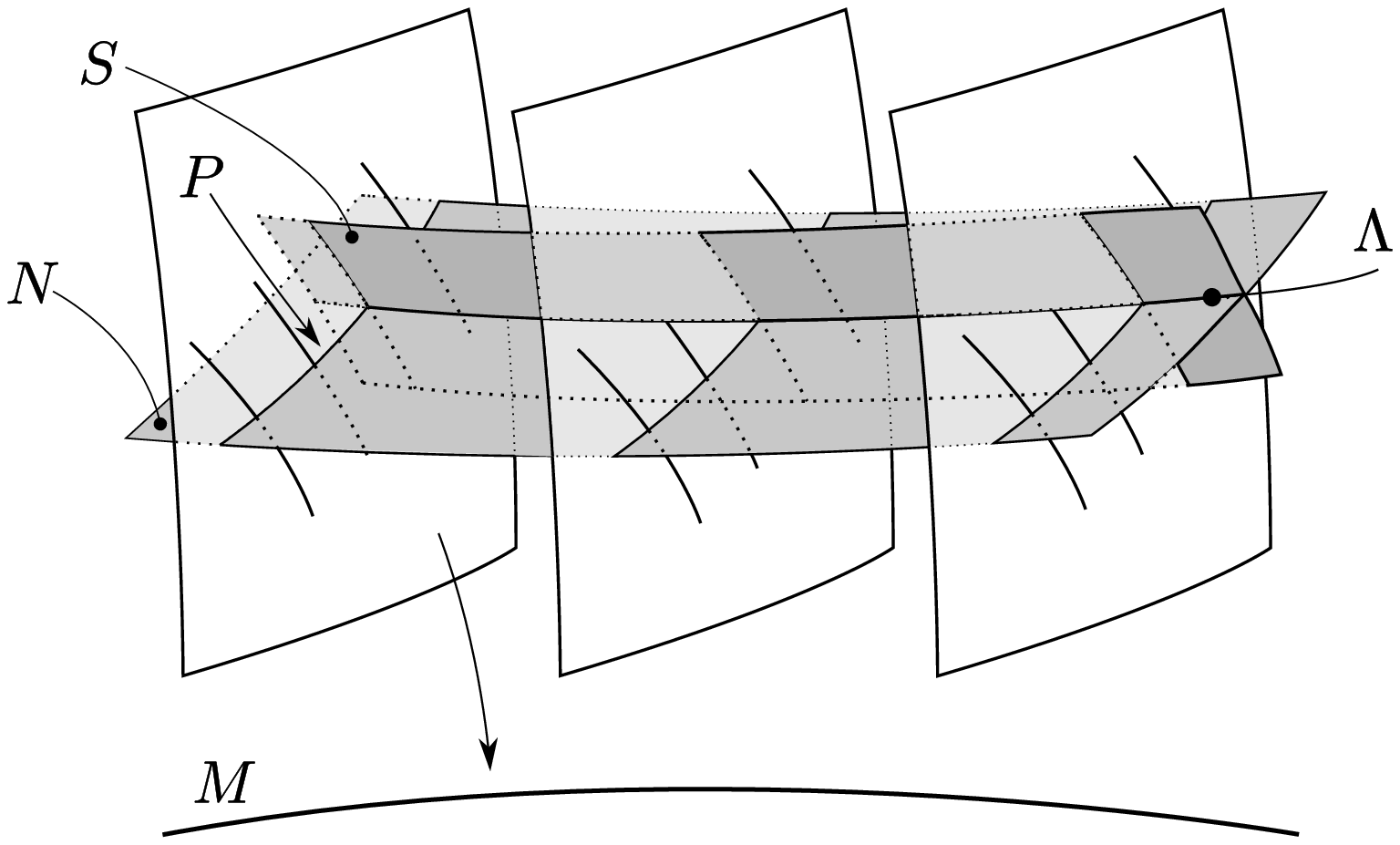}
\label{double_fibration}
\end{figure}

\medskip
The next step is to construct a submanifold $\mathcal{S}\subset\mbox{Diff}\left(M\right)$ which is transverse to
the $\mathcal{G}_{\pi_{0}}$-orbits. We consider the subbundle $N\subset M\times M$ defined at the beginning of the
section, and the associated tubular neighbourhood $U\overset{P}{\rightarrow}N$ of Lemma
\ref{lem-tubular-neighbourhood-N}. Then we denote by $\Delta\subset M\times M$ the diagonal and by $\iota_{\Delta}$ the
associated inclusion. Clearly, $\Delta$ lies actually in $N$. Then we define $S=P^{-1}\left(\Delta\right)$. It is a
submanifold of $M\times M$, but in fact will see that it is even a subbundle of $M\times M$ transverse to $N$.

To prove that it is a subbundle, first notice that $S$ is nothing but the total space of the induced bundle
$\iota_{\Delta}^*P$ over $\Delta$. Then, since the restriction to $\Delta$ of the first projection $P_M:M\times M$ is a
diffeomorphism $\Delta\to M$, one gets by composition a bundle $S\to M$. Finally, since $P_{M}\circ P=P_{M}$, the
projection of this last bundle is nothing but the restriction of $P_M$ to $S$, so that it is indeed a subbundle of
$M\times M$. Since $TS$ contains the vertical direction of the bundle $P$ over $N$, $S$ is transverse to $N$.


\medskip
Consequently, the set of sections of $S$ or more precisely
\[
\mathcal{S}=\left\{ \phi\in\mbox{Diff}\left(M\right)\mid\hat{\phi}\in\Gamma\left(M,S\right)\right\} ,\]
 is a Fréchet submanifold of $\mbox{Diff}\left(M\right)$. The following
characterisation will be useful.
\begin{lemma}
\label{lem-S-equivalent-property}A diffeomorphism $\phi$ lies in $\mathcal{S}$ if and only if
$P\circ\hat{\phi}=\hat{\Id}$.\end{lemma}
\begin{proof}
First, the left composition of a section $\hat{\phi}\in\Gamma\left(M,M\times M\right)$ by $P$ is still a section.
Therefore, $P\circ\hat{\phi}$ equals $\hat{\Id}$ if and only if its image $P\circ\hat{\phi}\left(M\right)$ coincides
with the image of $\hat{\Id}$, namely $\Delta$. But this happens precisely when the image of $\hat{\phi}$ lies in $S$,
by definition of $S$.
\end{proof}

\subsection{\label{sub-Frechet-bundle-chart}The principal bundle charts }

We now have all the technical tools for defining the principal bundle charts on $\mbox{Diff}\left(M\right)$. We first
define a chart near $\Id$. We consider the open set $U$ from Lemma \ref{lem-tubular-neighbourhood-N} and define the
set\[ \mathcal{U}=\left\{ \phi\in\mbox{Diff}\left(M\right)\mid\mbox{im}\left(\hat{\phi}\right)\subset U\right\} \]
 which is an open neighbourhood of $\Id$ in $\mbox{Diff}\left(M\right)$.
It is also $\mathcal{G}_{\pi_{0}}$-invariant because of the corresponding property for $U$.

On the other hand, in the definition of the principal bundle chart below, we will need that the composition
$P\circ\hat{\phi}$ lies in $\widehat{\mbox{Diff}}\left(M\right)$, or equivalently that the second factor
$P_{2}\circ\hat{\phi}$ is a diffeomorphism. We thus have to restrict to the smaller set \[ \tilde{\mathcal{U}}=\left\{
\phi\in\mathcal{U}\mid P_{2}\circ\hat{\phi}\in\Diff(M)\right\} ,\] which is open since the left composition by $P_2$ is
a Fréchet smooth map. The set $\tilde{\mathcal{U}}$ turns out to be also $\mathcal{G}_{\pi_{0}}$-invariant. Indeed, for
any $\phi\in\tilde{\mathcal{U}}$ and any $\psi\in\mathcal{G}_{\pi_{0}}$, we compute
$P_{2}\circ\widehat{\phi\circ\psi}$. This gives \[
P_{2}\circ\left(\Id,\phi\circ\psi\right)=P_{2}\circ\left(\psi,\phi\circ\psi\right)\] because of the second property of
$P$ given in Lemma \ref{lem-tubular-neighbourhood-N}. But this is equal to $P_{2}\circ\hat{\phi}\circ\psi$ which is a
diffeomorphism of $M$ since both $P_{2}\circ\hat{\phi}$ and $\psi$ are so.
\begin{lemma}
\label{lemma-chart}The map\begin{eqnarray*}
\Phi:\tilde{\mathcal{U}} & \longrightarrow & \mathcal{S}\times\mathcal{G}_{\pi_{0}}\\
\phi & \longmapsto & \left(\phi_{\mathcal{S}},\psi\right),\end{eqnarray*}
 where $\hat{\psi}=P\circ\hat{\phi}$ and $\phi_{\mathcal{S}}=\phi\circ\psi^{-1}$, is well-defined and is a Fréchet smooth diffeomorphism,
 whose inverse is \[ \Phi^{-1}:\left(\phi_{\mathcal{S}},\psi\right)\longmapsto\phi_{\mathcal{S}}\circ\psi.\]

\end{lemma}
\begin{proof}
First we check that this definition makes sense. Since the image of $P$ lies in $N$, it follows that $\hat{\psi}$ is a
section of $N$ and, thanks to Lemma \ref{lem-section-N-equivalent-vert}, that $\psi\in\mathcal{G}_{\pi_{0}}$. It
remains to check that $\phi_{\mathcal{S}}$ lies indeed in $\mathcal{S}$. According to Lemma
\ref{lem-S-equivalent-property}, we only need to check that $P\circ\hat{\phi}_{\mathcal{S}}=\hat{\Id}$, or equivalently
that $P_{2}\circ\hat{\phi}_{\mathcal{S}}=\Id$. The composition $P_{2}\circ\hat{\phi}_{\mathcal{S}}$ equals
$P_{2}\circ\left(\Id,\phi\circ\psi^{-1}\right)$. Now, if we use the property of $P_{2}$ given in Lemma
\ref{lem-tubular-neighbourhood-N}, we obtain $P_{2}\circ\left(\psi^{-1},\phi\circ\psi^{-1}\right)$ and thus
$P_{2}\circ\hat{\phi}\circ\psi^{-1}$. But this is exactly $\psi\circ\psi^{-1}$ and thus $\Id$. The map $\Phi$ is
therefore well defined. It is also smooth since it is made of compositions and inversion of diffeomorphisms, which are
both Fréchet smooth \cite[Exp. 4.4.5 and 4.4.6]{hamilton}.

Now, the image of the claimed inverse $\Phi^{-1}$ is included in $\tilde{U}$. Indeed, since $S$ is included in $U$,
then $\mathcal{S}$ is included in $\mathcal{U}$. It is moreover in $\tilde{\mathcal{U}}$ thanks to Lemma
\ref{lem-S-equivalent-property}. Finally, $\tilde{\mathcal{U}}$ is $\mathcal{G}_{\pi_{0}}$-invariant, hence the image
of $\Phi^{-1}$ is included in $\tilde{U}$.

Then, it is straightforward to see that the claimed inverse $\Phi^{-1}$ is a left inverse of $\Phi$. Let us check that
it is also a right inverse. For any $\phi_{\mathcal{S}}\in \mathcal{S}$ and any $\psi\in\mathcal{G}_{\pi_{0}}$ we have
to compute $(\phi_{\mathcal{S}}',\psi')=\Phi\left(\phi_{\mathcal{S}}\circ\psi\right)$. First, we have
$\psi'=P_{2}\circ\widehat{\phi_{\mathcal{S}}\circ\psi}$, i.e.,\[
\psi'=P_{2}\circ\left(\Id,\phi_{\mathcal{S}}\circ\psi\right).\] But the property of $P_{2}$ given in Lemma
\ref{lem-tubular-neighbourhood-N} shows that this equals $P_{2}\circ\hat{\phi}_{\mathcal{S}}\circ\psi$. On the other
hand, since $\hat{\phi}_{\mathcal{S}}$ lies in $\mathcal{S}$, one has $P_2\circ\hat{\phi}_{\mathcal{S}}=\Id $, hence
$\psi'=\psi$. Therefore, $\Phi\left(\phi_{\mathcal{S}}\circ\psi\right)$ is equal to
$\left(\phi'_{\mathcal{S}},\psi\right)$ with $\phi'_{\mathcal{S}}$ given by $\phi_{\mathcal{S}}\circ\psi\circ\psi^{-1}$
which thus coincides with $\phi_{\mathcal{S}}$. We have thus proved that the map
$\left(\phi_{\mathcal{S}},\psi\right)\mapsto\phi_{\mathcal{S}}\circ\psi$ is indeed the (double-sided) inverse of
$\Phi$.
\end{proof}

Therefore, near $\Id\in\mbox{Diff}\left(M\right)$, we have constructed a diffeomorphism $\Phi$ from a
$\mathcal{G}_{\pi_{0}}$-invariant neighbourhood $\tilde{\mathcal{U}}$ of $\Id$ to
$\mathcal{S}\times\mathcal{G}_{\pi_{0}}$. It is actually $\mathcal{G}_{\pi_{0}}$-equivariant, as one can easily check
on the inverse $\Phi^{-1}$.

Then, near any $\phi\in\mbox{Diff}\left(M\right)$, we can construct a similar diffeomorphism \[
\Phi_{\phi_{0}}:\tilde{\mathcal{U}}_{\phi_{0}}\longrightarrow\mathcal{S}_{\phi_{0}}\times\mathcal{G}_{\pi_{0}},\] where
$\tilde{\mathcal{U}}_{\phi_{0}}$ and $\mathcal{S}_{\phi_{0}}$ are obtained respectively from $\tilde{\mathcal{U}}$ and
$\mathcal{S}$ by left composition with $\phi_{0}$. The map $\Phi_{\phi_{0}}$ is simply defined by \[
\Phi_{\phi_{0}}\left(\phi\right)=\left(\phi_{0}\circ\Phi^{\mathcal{S}}\left(\phi_{0}^{-1}\circ\phi\right),\Phi^{\mathcal{G}_{\pi_{0}}}\left(\phi_{0}^{-1}\circ\phi\right)\right),\]
where $\Phi^{\mathcal{S}}$ and $\Phi^{\mathcal{G}_{\pi_{0}}}$ are respectively the $\mathcal{S}$ and
$\mathcal{G}_{\pi_{0}}$ component of $\Phi$.

Finally, it follows from Lemma \ref{lemma trivialisation->principal bundle} that the maps $\Phi_{\phi_{0}}$ form a
principal bundle atlas for $\mbox{Diff}\left(M\right)$ with structure group $\mathcal{G}_{\pi_{0}}$.

\section{The quotient space $\Diff_0(M)/(\mathcal{G}_{\pi_0}\cap\Diff_0(M))$}\label{section fibrations as fibers}

In this section, we prove the second part of Theorem \ref{theorem diff/fib}.

\begin{proof}First recall that, as already mentioned in the introduction, $\Fib(M,B)$ is an open subset of $C^{\infty}(M,B)$. Let $\pi_0\in\Fib(M,B)$. The smooth map
$$\Psi:\Diff_0(M)\to\Fib(M,B),\quad\phi\mapsto\pi_0\circ\phi^{-1}$$
induces a map
$$\tilde{\Psi}:\Diff_0(M)/(\mathcal{G}_{\pi_0}\cap\Diff_0(M))\to\Fib(M,B).$$
In order to prove that $\tilde{\Psi}$ is a diffeomorphism onto the component of $\pi_0$, we will prove that it is
injective, smooth, that it admits a local smooth inverse near any point and finally that its image is the connected
component of $\pi_0$.

\medskip Proving the injectivity is easy: for any two diffeomorphisms $\phi_1,\phi_2\in\Diff_0(M)$ satisfying
$\pi_0\circ\phi_1^{-1}=\pi_0\circ\phi_2^{-1}$, then $\phi_1^{-1}\circ\phi_2$ obviously belongs to $\mathcal{G}_{\pi_0}$
and thus $\phi_1$ and $\phi_2$ represent the same element in $\Diff_0(M)/(\mathcal{G}_{\pi_0}\cap\Diff_0(M))$.

\medskip The fact that $\tilde{\Psi}$ is smooth follows from the first part of Theorem \ref{theorem diff/fib}. Indeed, for any $\varphi$ in the
quotient $\Diff_0(M)/\mathcal{G}_{\pi_0}$, the first part of Theorem \ref{theorem diff/fib} implies that there exists a
smooth section $\sigma$ of the bundle $\Diff_0(M)\to\Diff_0(M)/\mathcal{G}_{\pi_0}$ defined on a neighbourhood of
$\varphi$. On this neighbourhood, $\tilde{\Psi}$ can be written $\tilde{\Psi}=\Psi\circ\sigma$ and therefore is smooth.

\medskip To construct a local inverse to $\tilde{\Psi}$, we will make use of the following lemma which relates the spaces of sections of two different fibrations of the same manifold.
\begin{lemma}[Roy \cite{roy_11}]\label{lemme de roy} Let $X$ be a smooth manifold and $p_1,p_2:X\to M$ be two fibrations
over $M$. Suppose there exists a common section $s_0:M\to X$, i.e., a smooth map satisfying
$$p_1\circ s_0=p_2\circ s_0=\Id_M.$$
Then, there exists open subsets $V_1\subset\Gamma(p_1)$ and $V_2\subset\Gamma(p_2)$ containing $s_0$ and such that the
map
$$V_1\to V_2,\quad \alpha\mapsto\alpha\circ(p_2\circ\alpha)^{-1}$$
is a diffeomorphism.

\end{lemma}

\begin{figure}[!h]
\centering
\includegraphics{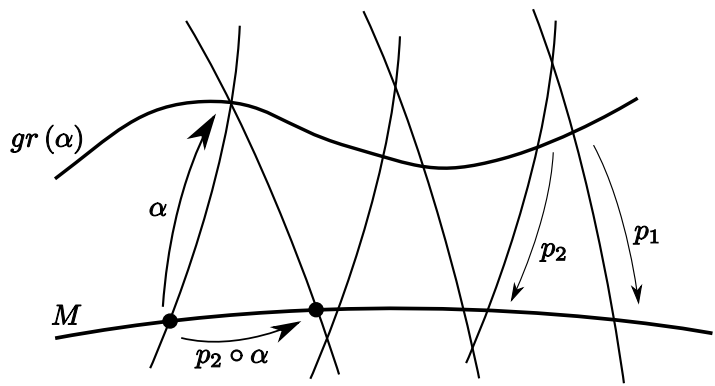}
\caption{On this picture, the graph of $s_0$ is identified with $M$} \label{double_fibration}
\end{figure}

For the sake of completeness, we briefly recall its proof.
\begin{proof}First of all, the map $A:\Gamma(p_1)\to C^{\infty}(M,M)$, $\alpha\mapsto p_2\circ\alpha$ is smooth and in particular continuous so that the
pre-image $V_1$ of the open subset $\Diff(M)\subset C^{\infty}(M,M)$ is an open subset of $\Gamma(p_1)$. Moreover,
$s_0\in V_1$ since $A(s_0)=\Id$. The map $\alpha\mapsto\alpha\circ(p_2\circ\alpha)^{-1}$ is well-defined on $V_1$ and
smooth, because built of compositions and inversion of diffeomorphisms. On the other hand, the image of $\alpha$ is indeed a section of $p_2$, since $p_2\circ\alpha\circ(p_2\circ\alpha)^{-1}=\Id$.

Reversing the roles of $p_1$ and $p_2$ one construct similarly a map $\beta\mapsto\beta\circ(p_1\circ\beta)^{-1}$,
defined on some open set $V_2\subset\Gamma(p_2)$ containing $s_0$. It is then easily checked that, up to taking smaller
$V_1$ and $V_2$, this map and the previous one are inverse to each other.
\end{proof}

Let us now use Lemma \ref{lemme de roy} to construct a local inverse to $\tilde{\Psi}$ around $\pi_0$. Denote by
$p_M:M\times B\to M$ and $p_B:M\times B\to B$ the natural projections on the first and the second factor respectively, and $s_0:M\to M\times B$, $x\mapsto(x,\pi_0(x))$. It turns out that the graph of $s_0$ is a subbundle of $p_B$ precisely because $\pi_0:M \to \times B$ is a submersion. Therefore, according to Lemma \ref{lem-voisinage-pratique}, there exists a tubular neighbourhood
$X$ of the graph of $s_0$ whose projection $q$ is $p_B$-vertical.

We now define the two necessary fibrations for applying Lemma \ref{lemme de roy}. The first one $p_1:X\to M$  is simply
the restriction of $p_M$ to $X$. The second one $p_2:X\to M$ is given by $p_2=p_1\circ q$. Now, we see easily that
$s_0$, which is obviously a section of $p_1$, is also a section of $p_2$. This is because the projection $q$ acts as
the identity on the graph of $s_0$. Therefore, we can apply Lemma \ref{lemme de roy} and deduce that for any section
$s=(\Id,\pi)$ of $p_1$ close to $s_0$, there exists a diffeomorphism $f\in\Diff(M)$ (depending smoothly on $s$ and
hence on $\pi$), such that $s\circ f$ is a section of $p_2$. Note by the way that since $\Fib(M,B)$ is open in
$C^{\infty}(M,B)$, if $s$ is sufficiently close to $s_0$, then $\pi$ is a fibration.  On the other hand, by
construction any section of $p_2$ has the form $(g,\pi_0)$, with $g:M\to M$. Indeed, for any $x\in M$, the fibre
$p_2^{-1}(x)$ is contained in $M\times\{b\}$, with $b=\pi_0 (x)$. Thus, one has
$$(f,\pi\circ f)=s\circ f=(g,\pi_0).$$
This proves in particular the crucial property $\pi\circ f=\pi_0$.

We denote $[f]$ the image of a diffeomorphism $f\in\Diff_0(M)$ in the quotient
$\Diff_0(M)/(\mathcal{G}_{\pi_0}\cap\Diff_0(M))$. The map $\chi:\pi\mapsto [f]$ is our candidate for being the local
smooth inverse of $\tilde{\Psi}$ around $\pi_0$. Indeed, for any $\pi$ near $\pi_0$, one has
$$\tilde{\Psi}\left(\chi(\pi) \right) = \tilde{\Psi}\left([f] \right) = \Psi (f) =  \pi_0 \circ f^{-1}=\pi.$$
Conversely, for any $[f]$ close to $[\Id]$, $\chi\left(\tilde{\Psi}([f]) \right)$ is given by
$$ \chi\left(\Psi(f) \right) = \chi\left(\pi_0\circ f^{-1} \right) = [g],$$ where $g\in \Diff(M)$ satisfies $\pi_0\circ f^{-1} \circ g = \pi_0$. But this precisely means that $f^{-1} \circ g \in \mathcal{G}_{\pi_0}$, hence $[g]=[f]$. Now, around any element $\tilde{\Psi}(\phi)=\pi_0\circ\phi^{-1}$ in
the image of $\tilde{\Psi}$, we construct the local inverse of $\tilde{\Psi}$ by
$\pi\mapsto \phi \circ\left[\chi(\pi\circ\phi) \right]$. Here, the left composition of $\phi$ with a class $[f]$ is defined to be $[\phi \circ f]$, which makes sense because
 $\phi\circ  \left( f \circ \mathcal{G}_{\pi_0}\right)=\left(\phi\circ   f \right)\circ \mathcal{G}_{\pi_0}$. We leave to the reader to check that this is indeed a local inverse of $\tilde{\Psi}$.

\medskip
Let us now finish our proof. It remains to show that the image of $\tilde{\Psi}$ is exactly the connected component
of $\pi_0$ in $\Fib(M,B)$. First, the image of $\tilde{\Psi}$ is connected since $\Diff_0(M)$ is. Then, let $\pi_1$ be
an element in the connected component of $\pi_0$ in $\Fib(M,B)$, so that there exists a path $(\pi_t)_{t\in[0,1]}$
joining them. The compactness of $[0,1]$ allows us to find $0=t_0<\ldots<t_k=1$ such that for any $i=1,\ldots,k$,
$\pi_{t_i}$ belongs to the domain of the local inverse as constructed previously around $\pi_{t_{i-1}}$ (just replace
$\pi_0$ with $\pi_{t_{i-1}}$). Therefore, for any $i=1,\ldots,k$, there exists an element $\phi_i$ in $\Diff_0(M)$ such
that
$$\pi_{t_i}=\pi_{t_{i-1}}\circ\phi_i.$$
Thus $\pi_1=\pi_0\circ(\phi_k\circ\ldots\circ\phi_1)$ which implies that $\pi_1$ lies in the image of $\tilde{\Psi}$.
\end{proof}

\section{The principal bundle structure of $\Fib(M,B)$}\label{section fib is fibered}

This section is devoted to the proof of Theorem \ref{theorem fib is fibered}.

\begin{demo}
Let $\pi_0\in\Fib(M,B)$ and $\sigma$ be a global section of $\pi_0$, i.e., $\pi_0\circ\sigma=\Id_B$. We denote by
$\Fib_0(M,B)$ the connected component of $\pi_0$ in $\Fib(M,B)$ and consider the set $\mathcal{S}$ of all fibrations
for which $\sigma$ is a section:
$$\mathcal{S}=\{\pi\in\Fib_0(M,B)\,|\,\pi\circ\sigma=\Id_B\}.$$
\begin{figure}[!h]
\centering
\includegraphics[scale=0.6]{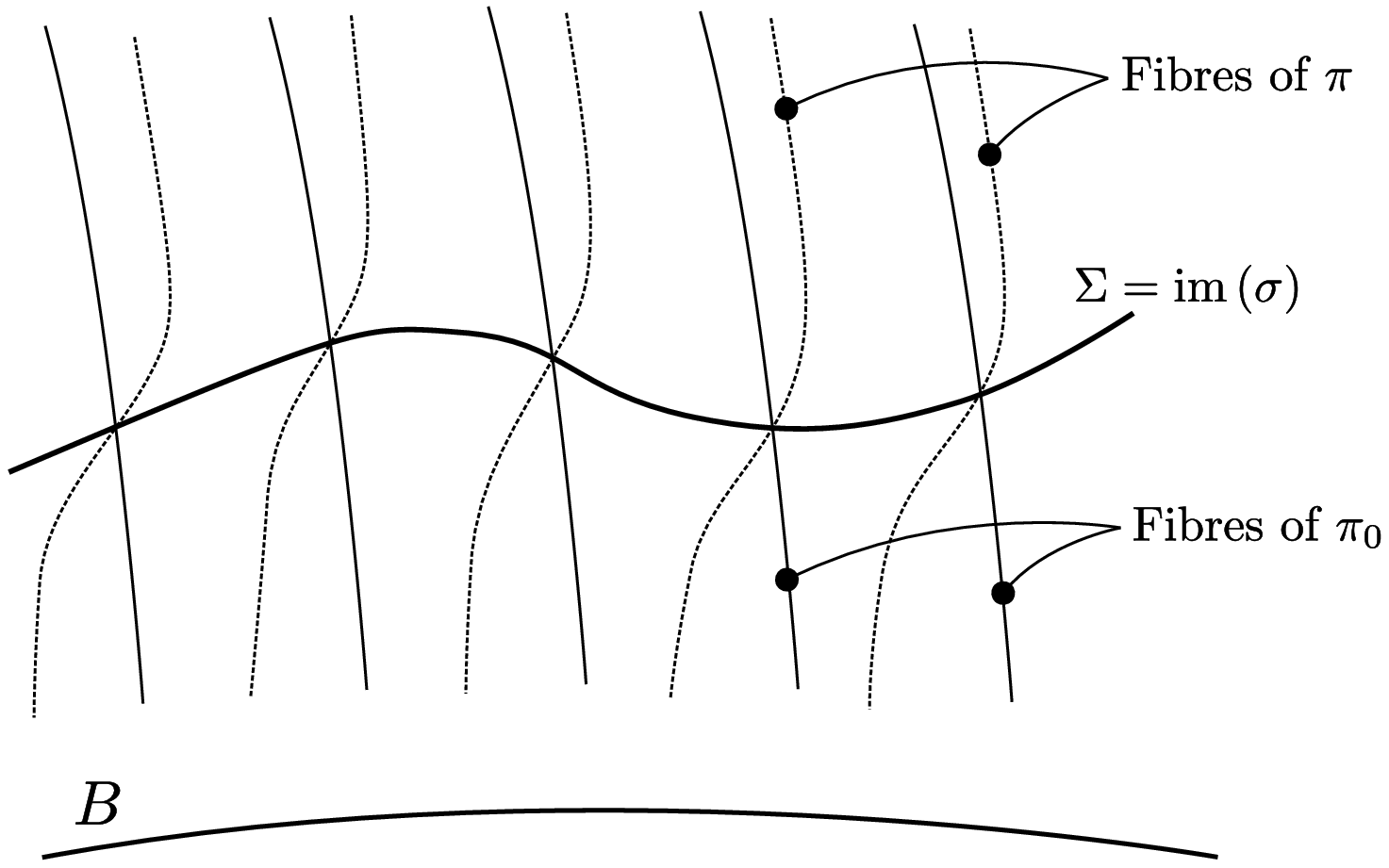}
\label{double_fibration}
\end{figure}

\medskip As a first step, let us prove that $\mathcal{S}$ is a Fréchet submanifold of $\Fib_0(M,B)$. To see this, first
remark that since $\Fib_0(M,B)$ is a homogeneous space (Theorem \ref{theorem diff/fib}), we can work in the
neighbourhood of $\pi_0$ without lack of generality. Then, remember that there exist an open set $\mathcal{U}$ in
$\Fib_0(M,B)$ containing $\pi_0$ and a Fréchet diffeomorphism $\Psi:\mathcal{U}\to \mathcal{V}$ onto a neighbourhood
$\mathcal{V}$ of the zero section in the Fréchet linear space of sections $\Gamma(\pi_0^*TB)$ of the vector bundle
$\pi_0^*TB$ over $M$. Moreover, this chart $\Psi$ can be chosen so that for all $x\in M$ and all $\pi\in \mathcal{U}$,
$\pi(x)=\pi_0(x)$ if and only if $\Psi(\pi)(x)=0$.

Now, if we denote by $\Sigma$ the submanifold of $M$ which is the image of the section $\sigma$, we see that for
$\pi\in \mathcal{U}$,
$$\pi\in\mathcal{S}\quad\Longleftrightarrow\quad\pi|_{\Sigma}=\pi_0|_{\Sigma}\quad\Longleftrightarrow\quad\Psi(\pi)|_{\Sigma}=0.$$
Let us consider the space $\Gamma_{\Sigma}(\pi_0^*TB)$ of sections of $\pi_0^*TB$ which vanish on $\Sigma$ and let us
show that it is a closed subspace of $\Gamma(\pi_0^*TB)$ with a closed complement.

\medskip First, $\Gamma_{\Sigma}(\pi_0^*TB)$ is obviously closed, so that it only remains to construct a closed
complement. To any vector field $\chi$ over $B$, we associate the section $\widetilde{\chi}$ of $\pi_0^*TB$ defined by
$$\widetilde{\chi}=\chi\circ\pi_0.$$
Now, we consider $\mathcal{F}=\{\widetilde{\chi}\,|\,\chi\in\Gamma(TB)\}$ the set of sections of $\pi_0^*TB$ which lift
vector fields on $B$. The set $\mathcal{F}$ is a closed linear subspace of $\Gamma(\pi_0^*TB)$. Moreover, for any
$s\in\Gamma(\pi_0^*TB)$, we can consider the vector field $s\circ\sigma$ over $B$. In the decomposition
$$s=\left(s-\widetilde{s\circ{\sigma}}\right)+\widetilde{s\circ{\sigma}},$$
the first term lies in $\Gamma_{\Sigma}(\pi_0^*TB)$ while the second is in $\mathcal{F}$, hence
$$\Gamma(\pi_0^*TB)=\Gamma_{\Sigma}(\pi_0^*TB)+\mathcal{F}.$$
Since moreover $\Gamma_{\Sigma}(\pi_0^*TB)$ and $\mathcal{F}$ have trivial intersection, the sum is direct and
$\mathcal{F}$ is a closed complement of $\Gamma_{\Sigma}(\pi_0^*TB)$. This proves our claim that
$\Gamma_{\Sigma}(\pi_0^*TB)$ is a closed subspace of $\Gamma(\pi_0^*TB)$, with closed complement.

\medskip
Consequently, the image of $\mathcal{S}\cap\mathcal{U}$ by $\Psi$ coincides with the intersection of $\mathcal{V}$ with
a closed complemented linear subspace of $\Gamma(\pi_0^*TB)$. This shows that $\mathcal{S}$ is a Fréchet submanifold of
$\Fib_0(M,B)$.

\medskip Let us now consider $\mathcal{W}=\{\pi\in\Fib_0(M,B)\,|\,\pi\circ\sigma\in\Diff_0(B)\}$. Since the right composition by
$\sigma$ is a continuous map, $\mathcal{W}$ is an open subset of $\Fib_0(M,B)$. It is moreover invariant under the
action of $\Diff_0(B)$ by composition on the left. We are going to show that there is an equivariant Fréchet
diffeomorphism
$$\mathcal{W}\,\simeq\,\Diff_0(B)\times\mathcal{S},$$
where the action of $\Diff_0(B)$ on the right hand side is given by composition on the left on the first factor.
Because of Lemma \ref{lemma trivialisation->principal bundle}, this will achieve the proof of Theorem \ref{theorem fib
is fibered}.

Let $\pi\in \mathcal{W}$ and set
$$\Phi(\pi)=(\pi\circ\sigma,(\pi\circ\sigma)^{-1}\circ\pi).$$
The first factor is obviously in $\Diff_0(B)$ and the second in $\mathcal{W}\cap\mathcal{S}$. Since composition and
inversion are smooth maps $\Phi$ is also a smooth map. Moreover, the fact that $\Phi$ is equivariant is easily checked:
for any $\phi\in\Diff_0(B)$, $\pi\in \mathcal{W}$, one has
\begin{align*} \Phi(\phi\circ\pi) &= (\phi\circ\pi\circ\sigma,
(\phi\circ\pi\circ\sigma)^{-1}\circ\phi\circ\pi)\\
&= (\phi\circ(\pi\circ\sigma),(\pi\circ\sigma)^{-1}\circ\pi).
\end{align*}
Finally, $\Phi$ has a smooth inverse given by the map $(\phi,\pi_S)\mapsto\phi\circ\pi_S$.
\end{demo}


\end{document}